\documentclass[a4paper, 11pt, reqno]{amsart}
\usepackage{amsfonts,amsmath,amssymb,amstext,amsxtra,euscript,mathrsfs}
\usepackage{enumitem}
\usepackage{color,graphicx,pgf,tikz}
\usetikzlibrary{arrows}
\usepackage{cancel}
\usepackage[normalem]{ulem}
\usepackage[polish,english]{babel}
\usepackage[T1]{fontenc}

\usepackage{newtxmath}
\usepackage{newtxtext}

\usepackage[compress, sort]{cite}

\usepackage[margin=2.5cm, centering]{geometry}
\usepackage[colorlinks,citecolor=blue,urlcolor=blue,bookmarks=true]{hyperref}
\hypersetup{
pdfpagemode=UseNone,
pdfstartview=FitH,
pdfdisplaydoctitle=true,
pdfborder={0 0 0}, 
pdftitle={Optimal bounds for the Dunkl kernel in the dihedral case},
pdfauthor={Jean-Philippe Anker and Bartosz Trojan},
pdflang=en-US
}

\newcommand{\pl}[1]{\foreignlanguage{polish}{#1}}
\newcommand{\fr}[1]{\foreignlanguage{french}{#1}}

\newcommand{\CC}{\mathbb{C}}

\newcommand{\RR}{\mathbb{R}}

\newcommand{\ZZ}{\mathbb{Z}}

\renewcommand{\Re}{\operatorname{Re}}
\renewcommand{\Im}{\operatorname{Im}}

\newcommand{\norm}[1]{\lvert {#1} \rvert}

\newcommand{\sprod}[2]{\langle {#1}, {#2} \rangle}

\renewcommand{\atop}[2]{\stackrel{{#1}}{{#2}}}

\definecolor{myred}{rgb}{0.77, 0.12, 0.23}
\definecolor{myblue}{rgb}{0.0, 0.45, 0.73}
\definecolor{mygreen}{rgb}{0.0, 0.56, 0.0}

\newcommand\closedchamber{\smash{\overline{\mathfrak{a}_+}}}

\renewcommand\epsilon{\varepsilon}
\newcommand\Id{\operatorname{Id}}
\renewcommand\Im{\operatorname{Im}}
\newcommand\kappamax{\kappa_{\ssf\max}}
\newcommand\kappamin{\kappa_{\ssf\min}}

\renewcommand\Re{\operatorname{Re}}
\newcommand\rme{\mathrm{e}}
\newcommand\rmi{\mathrm{i}}

\newcommand\ssf{\hspace{.3mm}}
\newcommand\ssb{\hspace{-.3mm}}
\newcommand\vsf{\hspace{.1mm}}

\theoremstyle{plain}
\newcounter{thm}

\newtheorem{theorem}{Theorem}

\newtheorem{corollary}[theorem]{Corollary}

\newtheorem{lemma}[theorem]{Lemma}

\newtheorem{remark}[theorem]{Remark}

\numberwithin{equation}{section}

\title{Optimal bounds for the Dunkl kernel in the dihedral case}

\author{Jean-Philippe Anker}
\address{
	\fr{
	Jean-Philippe Anker\\
	Institut Denis Poisson (UMR 7013),
	Universit\'e d'Orl\'eans,
	Universit\'e de Tours \& CNRS,
	B\^atiment de Math\'ematiques,
	B.P.~6759, 45067 Orl\'eans cedex 2, 
	France
	}
}
\email{anker@univ-orleans.fr}

\author{Bartosz Trojan}
\address{
	\pl{
	Bartosz Trojan\\
    Wydzia\l{} Matematyki\\
	Politechnika Wroc\l{}awska\\
	Wyb.~Wyspia\'{n}skiego 27\\
	50-370 Wroc\l{}aw\\
	Poland}
}
\email{bartosz.trojan@gmail.com}

\thanks{The second author acknowledges financial support from CNRS for a research trimester in 2023 in Orl\'eans.}

\begin{document}

\maketitle

\begin{abstract}
We establish sharp upper and lower estimates of the Dunkl kernel in the case of dihedral groups.
\end{abstract}

The heat kernel plays a central role in several branches of mathematics, e.g. in (harmonic) analysis and PDE, in Riemannian 
(and sub-Riemannian) geometry or in probability theory. In many applications it is desired to know sharp upper and lower
estimates in the largest possible space-time regime, see e.g. \cite{MR2569498} for a comprehensive presentation.
For Riemannian symmetric spaces of non-compact  type, the global behavior of the heat kernel was determined in 
\cite{AnkerJi1999} and \cite{AnkerOstellari2004}. In the present paper we are interested in the heat kernel arising in 
the rational Dunkl theory, which generalizes spherical Fourier analysis on Riemannian symmetric spaces of Euclidean type.
In this theory we deal with a finite root system $R$ in a Euclidean space $\mathfrak{a}$. The heat kernel is the fundamental
solution of the differential-difference equation
\[
	\partial_t h(t; x, y) = \Delta_x h(t; x, y) + 
	\sum_{\alpha \in R^+} \kappa(\alpha) \biggl( 2 \frac{\sprod{\alpha}{\nabla_x h(t; x, y)}}{\sprod{\alpha}{x}} + 
	|\alpha|^2 \frac{h(t; s_\alpha . x, y) - h(t; x, y)} {\sprod{\alpha}{x}^2} \bigg)  
\]
where $\Delta$ denotes the Laplacian on the underlying Euclidean space, $R_+$ the collection of positive roots, $\kappa$ 
the multiplicity function and $s_\alpha$ the orthogonal reflection with respect to $\alpha^\perp$.
The one dimensional case was carefully investigated in \cite{AnkerBensalemDziubanskiHamda2014}.
The general case has been recently studied in \cite{DziubanskiHejna2023} where the authors obtained the following estimates:
there is an explicit rational function $Q(t; x, y)$ and there are constants
$C_1, C_2 > 0$ and $c_1 > c_2 > 0$ such that
\begin{equation}\label{eq:0:1}
C_1 Q(t; x, y) g^{-c_1 \frac{d(x, y)^2}{t}} \leqslant h(t; x, y) \leqslant C_2 Q(t; x, y) e^{-c_2 \frac{d(x, y)^2}{t}}
\end{equation}
where
\[
d(x, y) = \min\big\{\norm{x - w . y} : w \in W\big\}
\]
denotes the orbital distance under the Weyl group action.

However, since the constants $c_1$ and $c_2$ are different, the estimates \eqref{eq:0:1} are \emph{not optimal}. 
Let us recall the expression
\[
h(t; x, y) = \frac{1}{c_k(2t)^{\gamma + \frac{N}{2}}} \, \exp\bigg\{-\frac{|x|^2 + |y|^2}{4t}\bigg\} \,
E\bigg(\frac{x}{\sqrt{2t}}, \frac{y}{\sqrt{2t}}\bigg)
\]
of the heat kernel in terms of the Dunkl kernel $E$, see \cite{Rosler1998}. The latter is an eigenfunction of all Dunkl 
operators, which are first order differential-difference operators.
In this paper we establish optimal estimates of $E$ (which imply optimal estimates of $h_t$) in the dihedral case $I_n$.

Our paper is organized as follows. The main result is stated in Theorem \ref{thm:1}. Its proof is carried out in Section \ref{sec:3}.
The overall strategy is explained in Section \ref{OverallStrategy} and the basic notation recalled in Section \ref{sec:1}.

\subsection*{Statement}
This work started as a joint project with J. Dziuba\'nski
which aimed at understanding the behavior of the heat kernel in the rational Dunkl setting
beyond the one dimensional case considered in \cite{AnkerBensalemDziubanskiHamda2014}.
In 2017 we obtained an upper bound of the Dunkl kernel for $A_2$ which was announced by J. Dziuba\'nski during his talk
at the conference ``Analysis and Applications'' organized in honor of E.M. Stein in September 2017 in Wroc{\l}aw
({\tt https://math.uni.wroc.pl/analysis2017}).
Later on, J. Dziuba\'nski and A. Hejna followed another approach
and obtained sharp upper and lower estimates for the heat kernel 
which, although not optimal, were sufficient for their needs (see \cite{DziubanskiHejna2023}).
Meanwhile we realized that our upper bound for $A_2$ was also a lower bound.
At the same time we obtained similar results for $B_2$.
Finally in June 2023, P. Graczyk and P. Sawyer informed us that
they were obtaining an upper and lower bound for $A_2$
by a completely different method, relying on a positive integral formula
$($see \cite{GraczykSawyer2023}$)$.

\section{Elements of rational Dunkl theory}
\label{sec:1}
In this section we introduce the necessary notation to define Dunkl kernels. For more details we refer to the pioneer paper 
\cite{Dunkl1989}, see also the surveys \cite{Rosler2003, Rosler2008}.

Let $R$ be a reduced (not necessarily crystallographic) finite root system in a $r$-dimensional Euclidean space $\mathfrak{a}$, 
that is, for each $\alpha \in R$, $s_\alpha(R) = R$ and $R \cap \RR \alpha = \{-\alpha, \alpha\}$, where 
\[
	s_\alpha (x) = x - \sprod{x}{\alpha\spcheck}\alpha
\]
and
\[
	\alpha\spcheck = \frac{2}{\sprod{\alpha}{\alpha}} \alpha.
\]
We fix a basis $\{\alpha_1, \alpha_2, \ldots, \alpha_r\}$ of simple roots in $R$. The corresponding subset of positive roots in $R$ will be
denoted by $R^+$. Let $W$ be the finite reflection group generated by $\{s_\alpha : \alpha \in R^+\}$. 
The fundamental domain for the action of $W$ on $\mathfrak{a}$ is the sector
\[
\mathfrak{a}_+ = \big\{x \in \mathfrak{a} : \sprod{\alpha}{x} > 0 \text{ for all } \alpha \in R^+ \big\}.
\]
Let $\kappa: R\rightarrow(0,\infty)$ be a positive multiplicity function on $R$, that is $\kappa$ is invariant under the action
of $W$ on $R$. The Dunkl operators, resp.~the Dunkl kernel $E(x,y)$ are $\kappa$-deformations of directional derivatives, 
resp.~of the exponential function $e^{\sprod{x}{y}}$. More precisely, the Dunkl operators are defined by
\begin{equation*}
T_\xi f(x) = \sprod{\xi}{\nabla f(x)} 
+ \sum_{\alpha \in R^+} \kappa(\alpha) \, \sprod{\alpha}{\xi} \, \frac{f(x) - f(s_\alpha x)}{\sprod{\alpha}{x}}
\quad \text{for every \,} \xi \in \mathfrak{a},
\end{equation*}
they commute pairwise and, for every $y\in\mathfrak{a}$, $E(\cdot,y)$ is the unique smooth solution to
\begin{equation}\label{eq:1}\begin{cases}
\,T_\xi f = \sprod{\xi}{y} f
\quad\text{for all } \xi\in\mathfrak{a},\\
\,f(0)=1.
\end{cases}\end{equation}
The Dunl kernel $E$ extends to a holomorphic function on $\mathfrak{a}_\CC \times \mathfrak{a}_\CC$, which satisfies
\begin{itemize}
	\item $E(x, y) = E(y, x)$;
	\item $E(\lambda x, y) = E(x, \lambda y)$, for all $\lambda \in \CC$;
	\item $E(w. x, w.y) = E(x, y)$, for all $w \in W$;
	\item $\overline{E(x, y)} = E(\overline{x}, \overline{y})$.
\end{itemize}
Moreover $E(x,y)>0$ when $x,y\in\mathfrak{a}$.
Our aim is to obtain sharp upper and lower estimates for $E$ on $\mathfrak{a} \times \mathfrak{a}$.

\section{The strategy}
\label{OverallStrategy}

In order to estimate the Dunkl kernel $E$, our strategy consists in using the differential-difference equations satisfied by $E$ and in constructing appropriate barrier functions. As an illustration, let us first consider the one dimensional case.

\subsection{Example\,: the one dimensional case}
Let us recall that $E: \RR \times \RR \rightarrow \CC$ satisfies
\[
	\left\{
		\begin{aligned}
		&\frac\partial{\partial x} E(x,y)
		+\frac \kappa x \big\{E(x,y)- E(x,- y)\big\}
		=y E(x,y)\\
		&\frac\partial{\partial x} E(x, -y)
		+\frac \kappa x \big\{E(x,- y)- E(x,y)\big\}
		=-y E(x, -y)
		\end{aligned}
		\quad\text{ for all } x, y \geqslant 0,
	\right.
\]
with the initial condition $E(x,y)=1$ whenever $x y = 0$. Given $c > 0$, we define
\begin{align*}
	\tilde{E}(x,y) &= (1 + c x y)^\kappa e^{-xy}E(x,y),\\
	\tilde{E}(x,-y) &= (1+c x y)^{\kappa+1} e^{-x y} E(x,-y),
\end{align*}
for all $x, y \geq 0$. Then
\begin{align}
	\label{eq:4}
	\frac\partial{\partial x} \tilde{E}(x,y)
	&=
	\frac \kappa {x (1+ c xy)}
	\{\tilde{E}(x, -y) - \tilde{E}(x,y)\},\\
	\label{eq:5}
	\frac\partial{\partial x} \tilde{E}(x, -y)
	&=\frac{\kappa (1+cxy)}x \{\tilde{E}(x,y)-\tilde{E}(x,-y)\}+\frac{(2\kappa+1)c+(\kappa c-2) c x y-2}
	{1+ c x y}y\tilde{E}(x,-y).
\end{align}
Let us fix $y > 0$. 

\vspace{1ex}
\noindent
\underline{Upper bound:}
Let $c$ be small, say $0<c<\frac2{2 \kappa+1}$. Hence,
\begin{equation}
	\label{eq:6}
	\tfrac{(2 \kappa+1)c+(\kappa c-2) c x y-2} {1+ c x y} <0.
\end{equation}
We claim that the function
\[
M(x)=\max\big\{\tilde{E}(x,y), \tilde{E}(x,-y)\big\}
\]
is a decreasing function on $[0,\infty)$. Indeed, on any interval $[a, b]$ where $\tilde{E}(x,y)\geqslant\tilde{E}(x,-y)$,
by \eqref{eq:4} we have $\tfrac\partial{\partial x}\tilde{E}(x,y)\leqslant0$. Hence
\[
\tilde{E}(x,-y)\leqslant\tilde{E}(x,y)\leqslant \tilde{E}(a,y)=M(a).
\]
Similarly, on any interval $[a, b]$ where $\tilde{E}(x,y) \leqslant \tilde{E}(x, -y)$, by \eqref{eq:6} and \eqref{eq:5}
we have $\tfrac\partial{\partial x} \tilde{E}(x,-y) \leqslant 0$. Hence
\[
	\tilde{E}(x,y) \leqslant \tilde{E}(x,- y) \leqslant \tilde{E}(a,- y)=M(a).
\]
In both cases, $M(x) \leqslant M(a)$ on $[a, b]$. Therefore, $M(x) \leqslant M(0) = 1$ on $[0,\infty)$.

\vspace{1ex}
\noindent
\underline{Lower bound:}
We argue similarly, assuming that $c$ is large, say $c > \tfrac{2}{\kappa}$. Hence,
\begin{equation}
	\label{eq:7}
	\tfrac{(2\kappa +1)c+(\kappa c-2)cxy-2}{1+ c xy}>0.
\end{equation}
The function
\[
	m(x)=\min\big\{\tilde{E}(x,y), \tilde{E}(x,-y) \big\}
\]
is an increasing function on $[0,\infty)$. Indeed, on any interval $[a, b]$ where $\tilde{E}(x,y) \leqslant \tilde{E}(x,-y)$,
by \eqref{eq:4} we have $\tfrac\partial{\partial x} \tilde{E}(x,y) \geqslant 0$. Hence
\[
	\tilde{E}(x,- y) \geqslant \tilde{E}(x,y) \geqslant \tilde{E}(a,y)=m(a).
\]
Similarly, on any interval $[a, b]$ where $\tilde{E}(x,y) \geqslant \tilde{E}(x,-y)$, by \eqref{eq:7} and \eqref{eq:5}
we have $\tfrac\partial{\partial x}\ssf\tilde{E}(x,-y) \geqslant 0$. Therefore,
\[
	\tilde{E}(x,y) \geqslant \tilde{E}(x, -y) \geqslant \tilde{E}(a, -y)=m(a).
\]
In both cases, $m(x) \geqslant m(a)$ on $[a, b]$. Hence $m(x) \geqslant m(0) = 1$ on $[0,\infty)$.

In conclusion we obtain the following global bound\,:
\[
	E(x,y) \approx
	\frac{e^{x y}}{(1+|x y|)^{\kappa}}
	\times
	\begin{cases}
		\,1 
		&\text{if $x y \geqslant 0$,}\\
		\,(1+|x y|)^{-1}
		&\text{if $x y \leqslant 0$.}\\
	\end{cases}
\]

\subsection{General root system}
\label{sec:2}
We set
\[
	\tilde{E}_w(x, y)
	=E(x, w.y) e^{-\langle x, y \rangle} 
	Q_w(x,y)
	\prod_{\alpha\in R^+} \big(1 + c \sprod{\alpha}{x} \sprod{\alpha}{y}\big)^{\kappa(\alpha)}
	\quad \text{for every \,} w\in W \text{ and }x, y \in \overline{\mathfrak{a}_+},
\]
where $c$ is a positive constant and the $Q_w$'s are barrier functions, to be determined, which are positive rational functions on $\overline{\mathfrak{a}_+}\times\overline{\mathfrak{a}_+}$.
Taking $\xi = x$ in \eqref{eq:1}, we get
\begin{align*}
	\sprod{x}{\nabla_x} 
	E(x, w. y) + \sum_{\alpha \in R^+} \kappa(\alpha)
	\big\{ E(x, w.y)- E(r_{\alpha}.x, w. y) \big\} 
	=\sprod{x}{w.y} E(x, w. y),
\end{align*}
thus
\begin{align*}	
	\frac{\sprod{x}{\nabla_x} \tilde{E}_w(x,y)} {\tilde{E}_w(x,y)}
	&=\sprod{x}{w. y} - \sprod{x}{y}
	+\frac{\sprod{x}{\nabla_x} Q_w(x,y)}{Q_w(x,y)}
	+\sum_{\alpha \in R^+} \kappa(\alpha)
	\frac{c \sprod{\alpha}{x} \sprod{\alpha}{y}}{1 + c \sprod{\alpha}{x} \sprod{\alpha}{y}} \\
	&\phantom{=}- \sum_{\alpha \in R^+} \kappa(\alpha)
	\bigg\{1 - \frac{Q_w(x,y)}{Q_{s_\alpha w}(x,y)}
	\frac{\tilde{E}_{s_\alpha w}(x,y)}{\tilde{E}_{w}(x,y)} \bigg\}.
\end{align*}
Hence
\begin{align*}
	\sprod{x}{\nabla_x} \tilde{E}_w(x, y)
	&=
	-\sum_{\alpha \in R^+} \kappa(\alpha)
	\frac{Q_w(x,y)}{Q_{s_\alpha w}(x,y)} \big\{\tilde{E}_w(x,y) - \tilde{E}_{s_\alpha w}(x,y) \big\}\\
	&\phantom{=}-\Lambda_w(x,y) \tilde{E}_{w}(x,y),
\end{align*}
where
\begin{equation}
	\label{Lambda_w}
	\begin{aligned}
	\Lambda_w(x,y)
	&=\sprod{x}{y - w. y} - \frac{\sprod{x}{\nabla_x} Q_w(x,y)} {Q_w(x,y)} \\
	&\phantom{=}
	+\sum_{\alpha \in R^+} \kappa(\alpha)
	\bigg\{\frac{1}{1 + c \sprod{\alpha}{x} \sprod{\alpha}{y}} 
	-\frac{Q_w(x,y)}{Q_{s_\alpha w}(x,y)} \bigg\}.
	\end{aligned}
\end{equation}
Our aim is to find positive rational functions $Q_w(x, y)$ and constants $c_+ \geqslant c_- > 0$ such that
\begin{equation}
	\label{eq:2}
	\Lambda_w(x, y) \geqslant 0, \quad \text{ for all } x, y \in \overline{\mathfrak{a}_+} \text{ and } c \in (0, c_-),
\end{equation}
and
\begin{equation}
	\label{eq:3}
	\Lambda_w(x, y) \leqslant 0, \quad \text{ for all } x, y \in \mathfrak{a}_+ \text{ and } c \in (c_+, \infty).
\end{equation}
Once \eqref{eq:2} and \eqref{eq:3} are achieved, we deduce as in the one dimensional case that
\begin{itemize}
\item
$M(t)=\max_{w\in W}\tilde{E}_w(tx,y)$ is a decreasing function of $t\in[0,\infty)$ when $c\in(0,c_-)$,
\item
$m(t)=\min_{w\in W}\tilde{E}_w(tx,y)$ is an increasing function of $t\in[0,\infty)$ when $c\in(c_+,\infty)$,
\end{itemize}
and we conclude that
\[
M(1)\leqslant M(0)=1=m(0)\leqslant m(1)\,.
\]

Currently we are able to complete this program for dihedral root systems, which are all (non necessarily crystallographic)
$2$-dimensional irreducible root systems. For general root systems we intend to return to the problem in the future.

\section{The Dunkl kernel for dihedral root systems}
\label{sec:3}
In this section we establish optimal bounds for the Dunkl kernel in the dihedral case, which includes in particular the root system $A_2$ considered in \cite{GraczykSawyer2023}.

\subsection{Dihedral root systems}
\label{sec:4}
Let us start by introducing the necessary notation. Let $I_n$, $n \geqslant 3$, be the root system in $\mathfrak{a} = \RR^2$
consisting of vectors
\[
	\Big\{ \pm\big(\cos \tfrac{\pi j}n,\sin\tfrac{\pi j}n\big) : j \in \{0, 1, \ldots, n-1\}\Big\}.
\]
Let us observe that the ends of the vectors in $I_n$ represent the vertices of the regular $(2n)$-gon in $\RR^2$. We 
set
\[
	\alpha_j = \big(\cos \tfrac{\pi j}n,\sin\tfrac{\pi j}n\big), \quad \text{ for } j \in \{0, 1, \ldots, n-1\}.
\]
Then $\{\alpha_0, \alpha_1, \ldots, \alpha_{n-1}\}$ is the set of positive roots in $I_n$. The positive Weyl chamber is
\[
	\mathfrak{a}_+ = \big\{ x \in \RR^2 : \sprod{\alpha_0}{x} > 0 \text{ and } \sprod{\alpha_{n-1}}{x} > 0 \big\}.
\]
The corresponding Weyl (or Coxeter) group has presentation
\[
	W = \langle r,s : r^n = \Id = s^2,  srs = r^{-1} \rangle.
\]
In fact, the group $W$ is the dihedral group which consists of $n$ rotations
\[
	r_j = r^{j} = 
	\begin{pmatrix}
		\cos\frac{2\pi j}n & -\sin\frac{2\pi j}n\\
		\sin\frac{2\pi j}n&\cos\frac{2\pi j}n
	\end{pmatrix}
	,
	\quad
	j \in \ZZ /n \ZZ,
\]
and $n$ reflections (or symmetries)
\[
	s_j =s r^{j} =
	\begin{pmatrix}
		-\cos\frac{2\pi j}n & \sin\frac{2\pi j}n\\
		\sin\frac{2\pi j}n  & \cos\frac{2\pi j}n
	\end{pmatrix}
	,
	\quad j \in \ZZ / n \ZZ.
\]
If $n$ is odd then all roots are in the same $W$-orbit while, if $n$ is even there are two $W$-orbits: 
$W\!.\alpha_0$ and $W\!.\alpha_1$. Thus, if $n$ is odd, we set $\kappa > 0$ to be the joint multiplicity of all roots while,
if $n$ is even, we let $\kappa_0 > 0$, respectively $\kappa_1 > 0$, to be the multiplicity of the roots in $W\!.\alpha_0$, 
respectively in $W\!.\alpha_1$. Set
\[
	\kappamin=
	\begin{cases}
		\kappa & \text{if $n$ is odd,} \\
		\min\{\kappa_0, \kappa_1\} & \text{if $n$ is even,}
	\end{cases}
	\quad\text{and}\quad
	\kappamax=\begin{cases}
		\kappa & \text{if $n$ is odd,} \\
		\max\{\kappa_0,\kappa_1\} & \text{if $n$ is even.}\\
	\end{cases}
\]
See Figure \ref{10gon} for a picture of $I_5$.

\begin{figure}[ht!]
	\includegraphics[width=120mm]{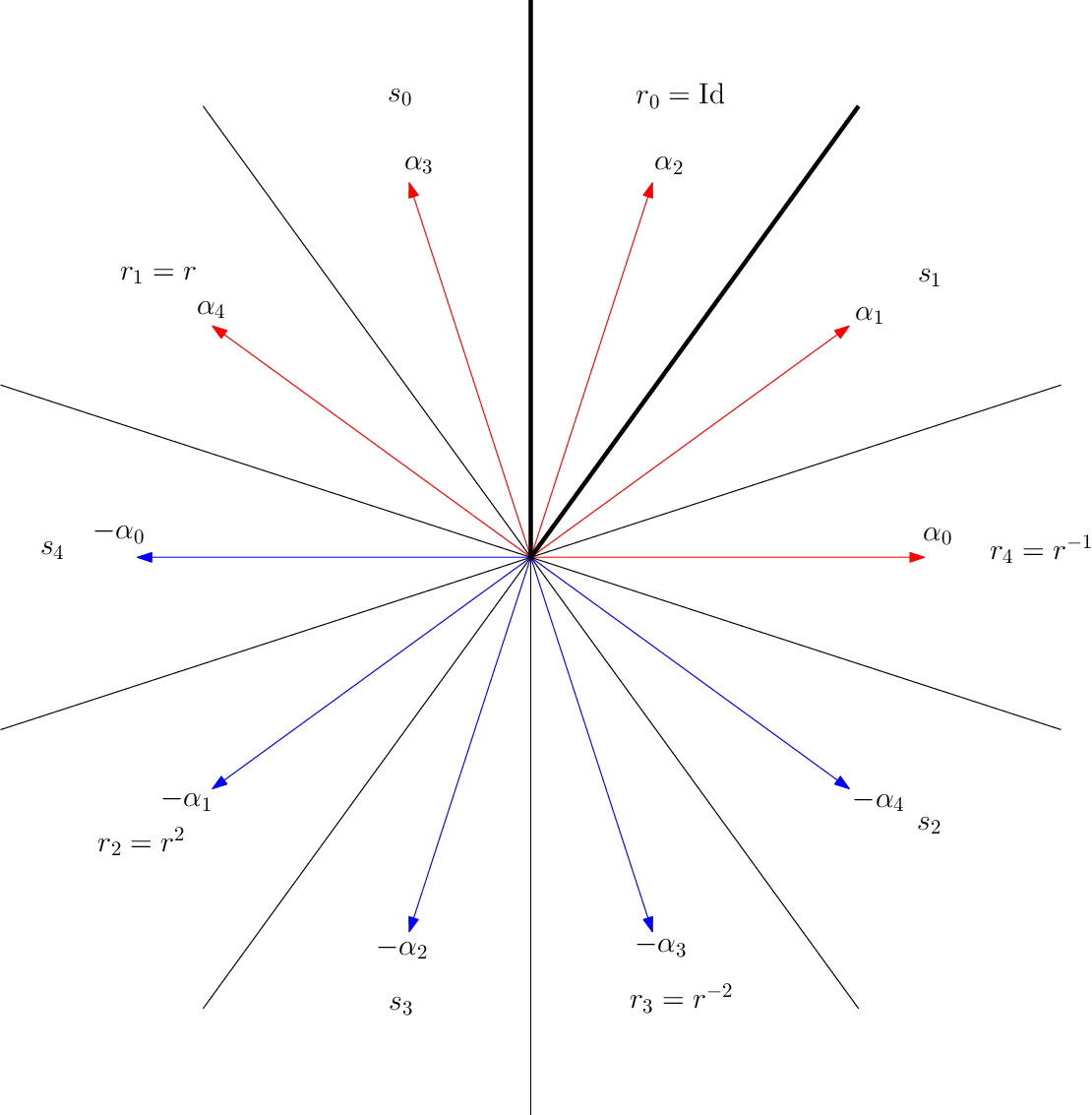}
	\caption{The case $n=5$}
	\label{10gon}
\end{figure}

Given $x,y \in \closedchamber$ and $j \in \ZZ/n\ZZ$, we define
\begin{alignat*}{2}
	\rho_j&=\sprod{x}{y - r_j.y}, \qquad &\sigma_j &= \sprod{x}{y - s_j. y}, \\
	D_{r_j}&=1 + c \rho_j,
	&D_{s_j}&=1 + c \sigma_j.
\end{alignat*}
Let
\begin{alignat*}{2}
	D_R&=\prod_{j \in \ZZ /n \ZZ} D_{r_j}, \qquad
	&D_S&= \prod_{j \in \ZZ/n\ZZ}D_{s_j}, \\
	\mathfrak{D}_R&=\sum_{j \in \ZZ /n \ZZ} \frac1{D_{r_j}},
	\qquad
	&\mathfrak{D}_{S}&=\sum_{j \in \ZZ/n\ZZ}\frac1{D_{s_j}},
\end{alignat*}
and
\[
	\mathfrak{S}=\prod_{j \in \ZZ /n \ZZ} \sigma_j.
\]
Let us observe that both $\rho_j$ and $\sigma_j$ are non-negative for every $j \in \ZZ /n \ZZ$. Moreover,  
$\rho_0 = 0$ and $D_{r_0}=1$ are trivial. The following two lemmas allows us to compare the $\rho_j$'s and $\sigma_j$'s 

\begin{lemma}
	\label{LemmaRhojSigmajInnerProduct}
	The following properties hold, for all $x,y \in \closedchamber$:
	\[
		\left\{
			\begin{aligned}
				\sigma_0 & = 2\sprod{\alpha_0}{x}\sprod{\alpha_0}{y}, \\
				\sigma_1 & = 2\sprod{\alpha_{n-1}}{x} \sprod{\alpha_{n-1}}{y}, \\
				\sigma_j &\approx \sprod{x}{y}, \quad \text{for all } j \in \{2, \ldots, n-1\},
			\end{aligned}
		\right.
		\quad\text{and}\quad
		\left\{
			\begin{aligned}
				\rho_1 &= 2 \big( \sin \tfrac\pi n\big) \sprod{\tilde{r}. x}{y}, \\
				\rho_{n-1} &= 2 \big( \sin  \tfrac\pi n \big) \sprod{x}{\tilde{r}. y}, \\
				\rho_j &\approx \sprod{x}{y}, \quad \text{for all } j \in \{2, 3, \ldots, n-2\},
			\end{aligned}
		\right.
	\]
	where $\tilde{r}$ is the rotation about the origin in $\RR^2$ by the angle $(\frac\pi2-\frac\pi n)$, that is
	\begin{equation}
		\label{eq:8}
		\tilde{r} = 
		\begin{pmatrix}
			\sin\frac\pi n &-\cos\frac\pi n \\
			\cos\frac\pi n&\sin\frac\pi n
		\end{pmatrix}.
	\end{equation}
\end{lemma}
\vspace{1mm}
\begin{proof}
	Since for each $j \in \{0, 1, \ldots, n-1\}$, $s_{-j}$ is the reflection with respect to $\alpha_j^\perp$, we get
	\[
		\sigma_{-j} = 2\sprod{\alpha_j}{x} \sprod{\alpha_j}{y}.
	\]
	If $0 < j < n-1$, then both $\sigma_{-j}$ and $\sprod{x}{y}$ reach (strictly) positive extrema, as $x$ and $y$ run through
	the compact set
	\[
		\text{S}(\closedchamber)=\big\{z \in \closedchamber : \norm{z} =1 \big\}.
	\]
	Hence, by homogeneity, we deduce that
	\[
		\sigma_{-j} \approx \sprod{x}{y}, \quad \text{ for all } x, y \in \closedchamber.
	\]
	Let us now turn to rotations. Notice that
	\[
		\Id- r_j
		=
		2 \big(\sin \tfrac{\pi j}{n} \big)
		\begin{pmatrix}
			\sin\frac{\pi j}n&\cos\frac{\pi j}n\\
			-\cos\frac{\pi j}n&\sin\frac{\pi j}n
		\end{pmatrix}
	\]
	where the matrix represents the rotation about the origin in $\RR^2$ by the angle $(\frac{\pi j}n - \frac\pi2)$.
	We deduce on the one hand that
	\[
		\left\{
		\begin{aligned}
			\rho_1 &= 2 (\sin \tfrac\pi n) \sprod{x}{(\tilde{r})^{-1}.y}, \\
			\rho_{n-1} &= 2 (\sin\tfrac\pi n) \sprod{x}{\tilde{r}. y}.
		\end{aligned}
		\right.
	\]
	On the other hand, for $1 < j < n-1$, the image $(\Id-r_j).\text{S}(\closedchamber)$ doesn't meet 
	$\text{S}(\closedchamber)$, thus arguing as for reflections we arrive at
	\[
		\rho_j \approx \sprod{x}{y}, \quad \text{ for all } x, y \in \closedchamber. \qedhere
	\]
\end{proof}

\begin{lemma}
	\label{LemmaRhojGeSigmak}
	For each $j \in \ZZ / n \ZZ \smallsetminus \{0\}$, there is $k \in \ZZ/n\ZZ$ such that 
	\[
		\rho_j\geqslant\sigma_k,
		\quad\text{ for all } x,y \in \closedchamber.
	\]
\end{lemma}
\begin{proof}
	Let $j \in \ZZ/n \ZZ \smallsetminus \{0\}$. As $r_{-j} \neq \Id$, there exists a positive root $\alpha_\ell$ such that 
	$r_{-j}.\alpha_\ell$ is negative. Let $k \equiv j - \ell \pmod n$, so that $s_{k-j}\ssb=\ssb s_{-\ell}$ is the reflection
	with respect to $\alpha_\ell^\perp$. Then
	\begin{align*}
		\rho_j-\sigma_k
		&=\sprod{x}{s_k. y - r_j. y} \\
		&=\sprod{x}{s_{k-j}( r_j. y) - r_j. y} \\
		&=-2 \underbrace{\sprod{\alpha_\ell}{x}}_{\geqslant 0}
		\underbrace{\sprod{\alpha_\ell}{r_j. y}}_{\leqslant 0}
		\geqslant 0.
		\qedhere
	\end{align*}
\end{proof}

Later we will also need the following lemma.

\begin{lemma}
	\label{LemmaExponentialSum}
	For $m_1, \ldots, m_k \in \ZZ/n\ZZ$, we set
	\begin{equation}
		\label{ExponentialSum}
		I(m_1,\dots,m_k)
		=
		\sum_{\atop{j_1, \ldots, j_k \in \ZZ/n\ZZ}{\mathrm{distinct}}}
		e^{\rmi\frac{2\pi}n(m_1j_1+\ldots+ m_k j_k)}.
	\end{equation}
	Then
	\begin{enumerate}[label=\rm (\alph*), ref=\alph*]
		\item
		\label{en:1:1}
		$I(m_1,\dots,m_k)$ is real valued and symmetric under permutations of $m_1, \ldots, m_k$;
		\item
		\label{en:1:2}
		$I(m_1,\dots,m_k)=
		\begin{cases}
			(n - k + 1) I(m_1, \ldots, m_{k-1}) & \text{if $m_k \equiv 0 \pmod n$}, \\
			-\sum_{1 \leqslant i < k} I(m_1, \ldots, m_i + m_k, \ldots, m_{k-1}) 
			&\text{if $m_k \not\equiv 0 \pmod n$};
		\end{cases}$
		\item
		\label{en:1:3}
		for every $m_1,\ldots, m_k \in \ZZ/n\ZZ$, there exists $c \in \ZZ$ such that
		\[
			I(m_1, \ldots, m_k)= c I(m_1 + \ldots + m_k);
		\]
		\item
		\label{en:1:4}
		for every $m \in \ZZ/n\ZZ$,
		\[
			I(m)=\begin{cases}
			n & \text{if $m \equiv 0 \pmod n$},\\
			0 & \text{if $m \not\equiv 0 \pmod n$}.
			\end{cases}
		\]
\end{enumerate}
\end{lemma}
\begin{proof}
	\eqref{en:1:4} is elementary.
	\eqref{en:1:1} is easily deduced from the definition \eqref{ExponentialSum}. Same for the first claim in~\eqref{en:1:2}.
	The second claim in \eqref{en:1:2} follows from the second claim in \eqref{en:1:4}.
	Finally \eqref{en:1:3} is deduced from \eqref{en:1:2} by induction.
\end{proof}

Next, in what follows we need the elementary symmetric polynomials $\rme_j$ on $n$ variables, namely,
\begin{equation}
	\left\{
	\begin{aligned}
	\rme_0(X_1, \ldots, X_n) &\equiv 1, \\
	\rme_k(X_1, \ldots, X_n) &= \sum_{1 \leqslant j_1 < \ldots < j_k \leqslant n} X_{j_1} X_{j_2} \cdots X_{j_k}, 
	\quad 1 \leqslant k < n\\
	\rme_n(X_1, \ldots, X_n) &= X_1 X_2 \cdots X_n.
	\end{aligned}
	\right.
\end{equation}
A straightforward argument shows that for $k \in \{1, \ldots, n-1\}$,
\[
	\rme_k(X_1, \ldots, X_n) = \frac{1}{k!} 
	\sum_{\atop{1 \leq j_1, \ldots, j_k \leq n}{\text{distinct}}} X_{j_1} X_{j_2} \cdots X_{j_k}.
\]
Recall that the elementary symmetric polynomials generate the algebra of symmetric polynomials in $n$ variables.

The following lemma allows us to compare symmetric polynomials in $\rho_j$'s and $\sigma_j$'s.
\begin{lemma}
	\label{FundamentalLemma}
	For each $k \in \{0, 1, \ldots, n-1\}$,
	\[
		\rme_k(\rho_0,\dots,\rho_{n-1})=\rme_k(\sigma_0,\dots,\sigma_{n-1})
	\]
	while $\rme_n(\rho_0,\dots,\rho_{n-1})=0$ and $\rme_n(\sigma_0,\dots,\sigma_{n-1})=\mathfrak{S}$.
\end{lemma}
\begin{remark}
	Let us observe that for all $k \in \{0, 1, \ldots, n-1\}$,
	$\rme_k(\rho_0,\rho_1,\dots,\rho_{n-1})=\rme_k(\rho_1,\dots,\rho_{n-1})$.
\end{remark}

\begin{proof}[Proof of Lemma \ref{FundamentalLemma}]
	First of all, the result is immediate when $k=0$ or $k=n$. Since
	\[
		\sum_{j \in \ZZ /n\ZZ}r_j = 0,
	\]	
	we deduce that $\rme_1(\rho_0,\dots,\rho_{n-1})$ and $\rme_1(\sigma_0,\dots,\sigma_{n-1})$ are equal to 
	$n\sprod{x}{y}$. Let us turn to the remaining cases where $1 < k < n$ and which are more involved. Let us consider the 
	linear operators
	\[
		S=\underbrace{s \otimes \dots \otimes s}_{\text{$k$ factors}}
	\]
	and
	\[
		R= 
		\sum_{\atop{j_1, \dots, j_k \in \ZZ/n \ZZ}{ \text{distinct}}}
		r_{j_1} \otimes\dots\otimes r_{j_k}
	\]
	acting on the space $\textrm{S}^k(\CC^2)$ of symmetric tensors in $\otimes^k\CC^2$. Notice that $S$ is an involution 
	which commutes with $R$. We claim that
	\begin{equation}
		\label{ClaimFundamentalLemma}
		R=SR=R S
	\end{equation}
	on $\textrm{S}^k(\CC^2)$. Let
	\begin{equation*}
		\textrm{S}^k(\CC^2)=E_{+1} \oplus E_{-1}
	\end{equation*}
	be the eigenspace decomposition of $S$. Then $R (E_{\pm1}) \subset E_{\pm1}$. Since, \eqref{ClaimFundamentalLemma} 
	trivially holds true on $E_{+1}$, it is enough to show that $R=0$ on $E_{-1}$.

	For the proof, we introduce suitable bases for our computations. For every $0 \leqslant \ell \leqslant k$, let us denote by
	$\tau_\ell$ the sum of tensors $v_1\!\otimes\dots\otimes v_k$ where $v_1,\dots,v_k$ are equal to $v_+ = (i, 1)$ or
	$v_-=(-i,1)$ with $\ell$ occurrences of $v_+$ and $(k-\ell)$ of $v_-$. Then $\{ \tau_\ell :  0 \leqslant \ell \leqslant k\}$
	is a basis of $\textrm{S}^k(\CC^2)$. Moreover, as $s(v_\pm) =  v_\mp$, we have $S(\tau_\ell) = \tau_{k-\ell}$.
	Hence $\{\tau_\ell + \tau_{k-\ell} : 0 \leqslant \ell \leqslant k/2\}$ is a basis of $E_{+1}$, and 
	$\{\tau_\ell-\tau_{k-\ell} : 0 \leqslant \ell < k/2 \}$ a basis of $E_{-1}$. Moreover, as
	\[
		r_j(v_\pm) = e^{\pm\rmi\frac{2\pi j}n}v_\pm,
	\]
	for each $j \in \ZZ /n \ZZ$, we obtain
	\begin{align*}
		R\big(\tau_\ell - \tau_{k-\ell}\big)
		&=R(\tau_\ell)- R(\tau_{k-\ell}) \\ 
		&= (\Re I_{k, \ell})(\tau_\ell - \tau_{k-\ell})+\rmi (\Im I_{k, \ell})(\tau_\ell + \tau_{k-\ell})
	\end{align*}
	for each $0 \leqslant \ell < k/2$ where $I_{k,\ssf\ell}$ denotes the exponential sum \eqref{ExponentialSum}
	with
	\[
		m_i=
		\begin{cases}
			1 & \text{$\ell$ times,}\\
			-1 & \text{$(k-\ell)$ times.}\\
		\end{cases}
	\]
	Notice that $m_1+\ldots+m_k=2\ell-k \not \equiv 0 \pmod n$, as $0 \leqslant 2 \ell < k < n$. By applying 
	\eqref{en:1:3} and \eqref{en:1:4} in Lemma \ref{LemmaExponentialSum}, we deduce that $I_{k, \ell} = 0$.

	In summary, $R$ vanishes on $E_{-1}$ and consequently \eqref{ClaimFundamentalLemma} holds true. By applying
	\[
		\sum_{\atop{j_1, \ldots, j_k \in \ZZ/n\ZZ}{\mathrm{distinct}}}
		r_{j_1} \otimes \ldots \otimes r_{j_k}
		=
		\sum_{\atop{j_1, \ldots , j_k \in \ZZ/n\ZZ}{\mathrm{distinct}}}
		s_{j_1}\otimes \ldots \otimes s_{j_k}.
	\]
	to $y\otimes \ldots \otimes y$ and taking the inner product with $x \otimes \ldots \otimes x$, we obtain
	\[
		\rme_k \big( \sprod{x}{r_0. y}, \ldots, \sprod{x}{r_{n-1}.y} \big)
		=
		\rme_k \big( \sprod{x}{s_0. y}, \ldots, \sprod{x}{s_{n-1}.y} \big)
	\]
	for every $1 < k < n$. Notice that this equality holds true for $k = 0$ and $k = 1$ too. We conclude that
	\begin{equation}
		\label{EqualityElementarySymmetricPolynomial}
		\rme_k\big(\rho_0, \ldots, \rho_{n-1}\big)
		=
		\rme_k\big(\sigma_0, \ldots, \sigma_{n-1}\big)
	\end{equation}	
	for every $0 < k < n$, by expressing both sides of \eqref{EqualityElementarySymmetricPolynomial} as the same polynomial in
	\[
		\rme_{k'} \big( \sprod{x}{r_0. y}, \ldots, \sprod{x}{r_{n-1}.y}\big)
		=
		\rme_{k'} \big(\sprod{x}{s_0. y}, \ldots, \sprod{x}{s_{n-1}.y}\big)
	\]
	with $0 \leqslant k' \leqslant k$ and the lemma follows.
\end{proof}

\begin{corollary}
	\label{DRDS}
	$D_R=D_S-c^n \mathfrak{S}$.
\end{corollary}
\begin{proof}
	On the one hand,
	\begin{equation*}
 		D_R=\sum_{0 \leqslant k < n} c^{k} \rme_k(\rho_0,\ldots,\rho_{n-1}),
	\end{equation*}
	and on the other hand,
	\begin{equation*}
		D_S=\sum_{0 \leqslant k \leqslant n}c^{k} \rme_k(\sigma_0,\ldots,\sigma_{n-1})
	\end{equation*}
	with $\rme_n(\sigma_0, \ldots,\sigma_{n-1})= \mathfrak{S}$. Then Lemma \ref{FundamentalLemma} allows us to conclude.
\end{proof}

\begin{corollary}
	\label{DRDRDSDS}
	$D_R\mathfrak{D}_{R}=D_S\mathfrak{D}_{S}$.
\end{corollary}
\begin{proof}
	This is obtained by expressing 
	\[
		D_R\mathfrak{D}_{R}=\sum_{0 \leqslant j < n} D_{r_0}\dots D_{r_{j-1}} D_{r_{j+1}} \dots D_{r_{n-1}}
	\]
	and
	\[
		D_S\mathfrak{D}_{S}=\sum_{0 \leqslant j< n}D_{s_0}\dots D_{s_{j-1}} D_{s_{j+1}} \dots D_{s_{n-1}}
	\]
	as the same polynomial in terms of
	\[
		\rme_k(\sigma_0, \ldots,\sigma_{n-1})=\rme_k(\rho_0,\ldots,\rho_{n-1})
	\]
	with $0 \leqslant k < n$.
\end{proof}

\subsection{Sharp estimates}
In this section we prove our main result. To do so, we introduce the following barrier functions:
\begin{equation}
	\label{DihedralQ}
	Q_w=
	\begin{cases}
		1 & \text{if } w=\Id, \\
		D_{s_j} &\text{if $w = s_j$ $(0 \leqslant j < n)$}, \\
		D_{r_j} \frac{D_S}{D_R} & \text{if $w= r_j$ $(0 < j < n)$}.\\
	\end{cases}
\end{equation}
We follow the overall strategy presented in Section \ref{OverallStrategy}.

Recall that $\Lambda_{\Id} \equiv 1$. We are going to determine the sign of $\Lambda_w$ ($w \neq \Id $) 
depending whether $c > 0$ is small or large.

\noindent
$\bullet$
\underline{Assume that} $w=s_j$ \underline{with} $j\in\ZZ/n\ZZ$. 
In this case \eqref{Lambda_w} writes
\begin{equation*}
	\Lambda_{s_j}
	=\sigma_j
	-\tfrac{c \sigma_j}{D_{s_j}}
	-\kappa(\alpha_j) 
	\bigl\{D_{s_j}-\tfrac1{D_{s_j}}\bigr\}
	-
	\sum_{\atop{k\in\ZZ/n\ZZ}{k \neq j}}
	\kappa(\alpha_k) 
	\Big\{\tfrac{D_{s_j}}{D_{r_{k-j}}}\tfrac{D_R}{D_S}-\tfrac1{D_{s_k}}\Big\},
\end{equation*}
hence
\begin{equation*}
	\Lambda_{s_j} \geqslant \sigma_j-A_j-\kappamax B_j
	\quad\text{and}\quad
	\Lambda_{s_j} \leqslant \sigma_j - \kappamin B_j,
\end{equation*}
where
\begin{equation*}
	A_j=\tfrac{c \sigma_j}{D_{s_j}}
	\quad\text{and}\quad
	B_j=D_{s_j}-\tfrac1{D_{s_j}}
	+\sum_{\atop{k \in \ZZ/n\ZZ}{k \neq j}} 
	\Big\{\tfrac{D_{s_j}}{D_{r_{k-j}}} \tfrac{D_R}{D_S} - \tfrac1{D_{s_k}}\Big\}.
\end{equation*}

\begin{lemma}
	\label{LemmaB}
	We have $B_j = c\sigma_j C_j$ with $2^{1-n}\leqslant C_j \leqslant n+1$.
\end{lemma}
\begin{proof}
	We have
	\begin{align*}
		B_j&=D_{s_j}-\tfrac1{D_{s_j}}
		+
		\tfrac{D_{s_j}D_R}{D_S} 
		\sum_{\atop{k \in \ZZ/n\ZZ}{k \neq 0}} 
		\tfrac1{D_{r_k}} - 
		\sum_{\atop{k\in\ZZ/n\ZZ}{k\neq j}} \tfrac1{D_{s_k}}\\
		&=\tfrac{D_{s_j}}{D_S} (D_S-D_R)
		+ \tfrac{D_{s_j}D_R}{D_S} \mathfrak{D}_{R}
		- \mathfrak{D}_{ S}.
	\end{align*}
	Using Corollaries \ref{DRDS} and \ref{DRDRDSDS}, we obtain
	\begin{align*}
		B_j =\tfrac{D_{s_j}}{D_S} c^{n} \mathfrak{S}
		+(D_{s_j}-1) \mathfrak{D}_{S}
		=c \sigma_j C_j
	\end{align*}
	where
	\begin{equation*}
		C_j= 
		\Big( \prod_{\atop{k \in \ZZ/n\ZZ}{k \neq j}} \tfrac{c \sigma_k}{D_{s_k}}\Big)
		+\mathfrak{D}_{S}.
	\end{equation*}
	To conclude, we notice that on the one hand $C_j < n + 1$. On the other hand, if 
	$c \sigma_k \geqslant 1$ for every $k \in \ZZ/n\ZZ$, then $C_j \geqslant 2^{1-n}$;
	otherwise, there is $k \in \ZZ/n\ZZ$ such that $c \sigma_k < 1$, thus $C_j > \tfrac1{D_{s_k}} >\tfrac12$ which is 
	$\geqslant 2^{1-n}$.
\end{proof}

\begin{corollary}
	We have
	\begin{enumerate}[label=\rm (\alph*), ref=\alph*]
		\item
		$\Lambda_{s_j} \geqslant 0$, if $0 < c \leqslant \tfrac1{\kappamax(n + 1)+1}$;
		\item
		$\Lambda_{s_j} \leqslant 0$, if $c \geqslant \tfrac{2^{n-1}}{\kappamin}$.
	\end{enumerate}
\end{corollary}
\begin{proof}
	On the one hand,
	\[
		\Lambda_{s_j} \geqslant \sigma_j \Big\{1-c-c \kappamax (n+1)\Big\},
	\]
	which is non-negative provided that $c \leqslant \tfrac1{\kappamax(n+1)+1}$. On the other hand,
	\[
		\Lambda_{s_j} \leqslant \sigma_j \Big\{1-c\kappamin 2^{1-n}\Big\},
	\]
	which is non-positive provided that $c \geqslant \tfrac{2^{n-1}}\kappamin$.
\end{proof}

\noindent
$\bullet$
\underline{Assume that} $w=r_j$ \underline{with} $j \in \ZZ/n\ZZ \smallsetminus \{0\}$.
In this case, \eqref{Lambda_w} writes
\begin{equation*}
	\Lambda_{r_j}
	=\rho_j
	\sum_{k \in \ZZ/n\ZZ} \tfrac{c \sigma_k}{D_{s_k}} + 
	\sum_{\atop{k \in \ZZ/n\ZZ \smallsetminus \{0\}}{k \neq j}}
	\tfrac{c \rho_k}{D_{r_k}}
	-\sum_{k \in \ZZ/n\ZZ}
	\kappa(\alpha_k) \Big\{ \tfrac{D_{r_j}}{D_{s_{k-j}}} \tfrac{D_S}{D_R} -\tfrac1{D_{s_k}} \Big\}
\end{equation*}
hence
\begin{equation*}
	\Lambda_{r_j} \geqslant \rho_j-\tilde{A}_j - \kappamax \tilde{B}_j
	\quad\text{and}\quad
	\Lambda_{r_j} \leqslant \rho_j - \kappamin \tilde{B}_j,
\end{equation*}
where
\begin{equation*}
	\tilde{A}_j=
	\sum_{k \in \ZZ/n\ZZ} 
	\tfrac{c\sigma_k}{D_{s_k}} 
	\sum_{\atop{k \in \ZZ/n\ZZ \smallsetminus \{0\}}{k \neq j}}
	\tfrac{c \rho_k}{D_{r_k}}
\end{equation*}
and
\begin{equation*}
	\tilde{B}_j= \sum_{k \in \ZZ/n\ZZ} 
	\Big\{\tfrac{D_{r_j}}{D_{s_{k-j}}} \tfrac{D_S}{D_R} - \tfrac1{D_{s_k}}\Big\}
	=\Big\{D_{r_j} \tfrac{D_S}{D_R} - 1 \Big\} \mathfrak{D}_{S}.
\end{equation*}

\begin{lemma}
	\label{LemmaTilde}
	We have
	\[
		\tilde{A}_j=\tfrac{c\rho_j}{D_{r_j}}+\tfrac{c^n \mathfrak{S}}{D_S} \mathfrak{D}_{R},
		\qquad\qquad
		\tilde{B}_j 
		=
		c \rho_j \mathfrak{D}_{R} + 
		\tfrac{c^{n} \mathfrak{S}}{D_S} \mathfrak{D}_{R},
	\]
	and
	\[
		1 \leqslant \mathfrak{D}_{R} \leqslant n.
	\]
\end{lemma}
\begin{proof}
	On the one hand, from
	\begin{equation*}
		\sum_{k \in \ZZ/n\ZZ} \tfrac{c \sigma_k}{D_{s_k}}=n-\mathfrak{D}_{S}
		\quad\text{and}\quad
		\sum_{k \in \ZZ/n\ZZ} \tfrac{c \rho_k}{D_{r_k}}=n-\mathfrak{D}_{R}
	\end{equation*}
	we deduce that
	\begin{align*}
		\tilde{A}_j
		&=\sum_{k \in \ZZ/n\ZZ} \tfrac{c \sigma_k}{D_{s_k}}
		- 
		\sum_{k \in \ZZ/n\ZZ} \tfrac{c\rho_k}{D_{r_k}} + \tfrac{c\rho_j}{D_{r_j}} \\
		&=\tfrac{c\rho_j}{D_{r_j}}+\mathfrak{D}_{R}-\mathfrak{D}_{S},
	\end{align*}
	hence by Corollaries \ref{DRDS} and \ref{DRDRDSDS},
	\begin{align*}
		\tilde{A}_j
		&=\tfrac{c \rho_j}{D_{r_j}}+\Big(1-\tfrac{D_R}{D_S}\Big)\mathfrak{D}_{R} \\
		&=\tfrac{c \rho_j}{D_{r_j}}+\tfrac{c^n\mathfrak{S}}{D_S}\mathfrak{D}_{R}.
	\end{align*}
	On the other hand,
	\begin{align*}
		D_{r_j} \tfrac{D_S}{D_R}-1
		&=c \rho_j \tfrac{D_S}{D_R}+\tfrac{D_S}{D_R}-1 \\
		&=c \rho_j \tfrac{D_S}{D_R}+\tfrac{c^{\vsf n} \mathfrak{S}}{D_R}.
	\end{align*}
	Now using again Corollaries \ref{DRDS} and \ref{DRDRDSDS}, we obtain 
	\begin{align*}
		D_{r_j}\tfrac{D_S}{D_R}-1
		&=c \rho_j \tfrac{D_S}{D_R}+\tfrac{D_S}{D_R}-1 \\
		&=c \rho_j \tfrac{D_S}{D_R}+\tfrac{c^{n} \mathfrak{S}}{D_R},
	\end{align*}
	and 
	\begin{align*}
		\tilde{B}_j
		&=c \rho_j \tfrac{D_S \mathfrak{D}_{S}}{D_R} 
		+c^{n}\mathfrak{S}\tfrac{\mathfrak{D}_{S}}{D_R} \\
		&=c \rho_j \mathfrak{D}_{R}
		+\tfrac{c^{n} \mathfrak{S}}{D_S} \mathfrak{D}_{R}.
	\end{align*}
	Finally, since
	\[
		D_{r_0} =1 \quad\text{and}\quad
		D_{r_k} \geqslant 1 \quad \text{for all } k \in \{1, 2, \ldots, n-1\},
	\]
	we have $1 \leqslant \mathfrak{D}_R \leqslant n$ and the lemma follows.
\end{proof}

\begin{corollary}
	We have
	\begin{enumerate}[label=\rm (\alph*), ref=\alph*]
		\item $\Lambda_{r_j} \geqslant 0$, if $0 < c \leqslant \tfrac 1{2 (1 + \kappamax) n}$,
		\item $\Lambda_{r_j} \leqslant 0$, if $c \geqslant \tfrac1\kappamin$.
	\end{enumerate}
\end{corollary}
\begin{proof}
	On the one hand,
	\[
		\Lambda_{r_j} \leqslant \rho_j \big\{1 - c \kappamin \big\},
	\]
	which is non-positive if $c \geqslant \tfrac1\kappamin$. On the other hand, according to Lemma \ref{LemmaRhojGeSigmak},
	there is $k \in \ZZ/n\ZZ$ such that $\rho_j \geqslant \sigma_k$, and so
	\begin{align*}
		\Lambda_{r_j}
		&\geqslant 
		\rho_j \big\{1-c(1+\kappamax n)\big\}
		-(1+\kappamax) n \tfrac{c^n \mathfrak{S}}{D_S}\\
		&=(\rho_j-\sigma_k) \big\{1-c(1+\kappamax n)\big\}
		+\sigma_k \big\{1-2 c (1+\kappamax) n \big\},
	\end{align*}
	which is non-negative provided that $0 < c \leqslant \tfrac1{2(1+\kappamax)n}$.
\end{proof}

In conclusion, we obtain the following global upper and lower bound for the Dunkl kernel.
\begin{theorem}
	\label{thm:1}
	For every $w \in W$ and $x, y \in \closedchamber$,
	\begin{align*}
		E(x,w.y) & 
		\approx
		e^{\sprod{x}{y}}
		\Big( 
		\prod_{\alpha \in R^+} (1 + \sprod{\alpha}{x} \sprod{\alpha}{y})^{-\kappa(\alpha)}\Big) \\ 
		&\times
		\begin{cases}
			1 
			& \text{if $w=\Id$,}\\[1.25ex]
			\frac1{1+\sprod{\alpha_0}{x} \sprod{\alpha_0}{y}}
			& \text{if $w = s_0$,}\\[1.25ex]
			\frac1{1+\sprod{\alpha_{n-1}}{x} \sprod{\alpha_{n-1}}{y}} 
			&\text{if $w = s_1$,}\\[1.25ex]
			\frac1{1+\sprod{x}{y}}  
			&\text{if $w = s_j$ with $1< j < n$,}\\[1.25ex]
			\frac{1+\sprod{x}{\tilde{r}.y}} 
			{(1+\sprod{\alpha_0}{x}\sprod{\alpha_0}{y})(1 + \sprod{\alpha_{n-1}}{x}\sprod{\alpha_{n-1}}{y})
			(1+\sprod{x}{y})} 
 			&\text{if $w = r_1$},\\[1.25ex]
			\frac{1+\sprod{\tilde{r}.x}{y}}
			{(1+\sprod{\alpha_0}{x} \sprod{\alpha_0}{y})
			(1 + \sprod{\alpha_{n-1}}{x} \sprod{\alpha_{n-1}}{y})
			(1 + \sprod{x}{y})}
			&\text{if $w = r_{n-1}$,}\\[1.25ex]
			\frac{(1+\sprod{\tilde{r}.x}{y})(1+\sprod{x}{\tilde{r}.y})}
			{(1 + \sprod{\alpha_0}{x} \sprod{\alpha_0}{y})
			(1+\sprod{\alpha_{n-1}}{x} \sprod{\alpha_{n-1}}{y})
			(1+\sprod{x}{y})^2}
			&\text{if $w = r_j$ with $1 < j < n - 1$,}
		\end{cases}
	\end{align*}
	where $\tilde{r}$ is defined in \eqref{eq:8}.
\end{theorem}
\begin{proof}
	According to the choice \eqref{DihedralQ}, we have proved that
	\begin{align*}
		E(x,w.y) \approx e^{\sprod{x}{y}}
		\Big(\prod_{\alpha \in R^+} (1+\sprod{\alpha}{x} \sprod{\alpha}{y})^{-\kappa(\alpha)}\Big)
		&\times
		\begin{cases}
			1 & \text{if $w = \Id$,}\\[1.25ex]
			\frac1{1+\sigma_j}
			&\text{if $w = s_j$ with $j \in \ZZ/n\ZZ$,}\\[1.25ex]
			\frac{\prod_{\atop{0 < k < n}{k \neq j}} (1+\rho_k)}
			{\prod_{0 \leqslant k < n} (1 + \sigma_k)} 
			&\text{if $w = r_j$ with $j \in \ZZ/n\ZZ \smallsetminus \{0\}$,}
		\end{cases}
	\end{align*}
and the theorem is a consequence of Lemma \ref{LemmaRhojSigmajInnerProduct}.
\end{proof}

\begin{bibliography}{dunkl}
	\bibliographystyle{amsplain}

\providecommand{\bysame}{\leavevmode\hbox to3em{\hrulefill}\thinspace}
\providecommand{\MR}{\relax\ifhmode\unskip\space\fi MR }
\providecommand{\MRhref}[2]{%
  \href{http://www.ams.org/mathscinet-getitem?mr=#1}{#2}
}
\providecommand{\href}[2]{#2}
\begin{thebibliography}{10}

\bibitem{AnkerBensalemDziubanskiHamda2014}
J.-Ph. Anker, N.~Ben~Salem, J.~Dziuba\'nski, and N.~Hamda, \emph{The {H}ardy
  space {$H^1$} in the rational {D}unkl setting}, Constr. Approx. \textbf{42}
  (2015), 93--128.

\bibitem{AnkerJi1999}
{J.-Ph.} Anker and L.~Ji, \emph{{H}eat kernel and {G}reen function estimates on
  noncompact symmetric spaces}, Geom. Funct. Anal. \textbf{9} (1999),
  1035--1091.

\bibitem{AnkerOstellari2004}
{J.-Ph.} Anker and P.~Ostellari, \emph{The heat kernel on symmetric spaces},
  Lie Groups and Symmetric Spaces: In Memory of F.I. Karpelevich (S.G.
  Gindikin, ed.), Amer. Math. Soc. Transl. (2), no. 210, American Mathematical
  Society, Providence, RI, 2004, pp.~27--46.

\bibitem{Dunkl1989}
C.F. Dunkl, \emph{Differential-difference operators associated to reflection
  groups}, Trans. Amer. Math. Soc. \textbf{311} (1989), 16--183.

\bibitem{DziubanskiHejna2023}
J.~Dziuba\'nski and A.~Hejna, \emph{Upper and lower bounds for the {D}unkl heat
  kernel}, Calc. Var. \textbf{62} (2023), no.~1, article 25.

\bibitem{GraczykSawyer2023}
P.~Graczyk and P.~Sawyer, \emph{A formula and sharp estimates for the {D}unkl
  kernel for the root system {$A_2$}}, arXiv:2308.01388, 2023.

\bibitem{MR2569498}
A.~Grigor'yan, \emph{Heat kernel and analysis on manifolds}, AMS/IP Studies in
  Advanced Mathematics, vol.~47, American Mathematical Society, Providence, RI;
  International Press, Boston, MA, 2009.

\bibitem{Rosler1998}
M.~R\"osler, \emph{Generalized {H}ermite polynomials and the heat equation for
  {D}unkl operators}, Comm. Math. Phys. \textbf{192} (1998), 519--542.

\bibitem{Rosler2003}
\bysame, \emph{Dunkl operators (theory and applications)}, Orthogonal
  polynomials and special functions {(Leuven, 2002)} (E.~Koelink and
  W.~Van~Assche, eds.), Lecture Notes Math., vol. 1817, Springer, 2003,
  pp.~93--135.

\bibitem{Rosler2008}
M.~R\"osler and M.~Voit, \emph{Dunkl theory, convolution algebras, and related
  {M}arkov processes}, Harmonic and stochastic analysis of Dunkl processes
  (P.~Graczyk, M.~R\"osler, and M.~Yor, eds.), Travaux en cours 71, Hermann,
  Paris, 2008, pp.~1--112.

\end{thebibliography}
\end{bibliography}

\end{document}